\documentclass[12pt,a4paper,twoside]{amsart}
\usepackage{amscd,amssymb,amsmath,amsfonts,amsthm}
\usepackage{a4}
\usepackage{color}
\usepackage{comment}

\usepackage{enumitem}

\usepackage[utf8]{inputenc}
\usepackage[english]{babel}

\newtheorem{theorem}{Theorem}[section]

\newtheorem{lemma}[theorem]{Lemma}

\newtheorem{proposition}[theorem]{Proposition}

\newtheorem{corollary}[theorem]{Corollary}

\theoremstyle{definition}

\newtheorem{definition}[theorem]{Definition}

\newtheorem{claim}[theorem]{Claim}

\newtheorem{conjecture}[theorem]{Conjecture}

\newcommand{\mA}{\mathbb A}

\newcommand{\mC}{{\mathbb C}}

\newcommand{\mE}{{\mathbb E}}

\newcommand{\mF}{\mathbb F}

\newcommand{\mQ}{\mathbb Q}

\newcommand{\mR}{{\mathbb R}}

\newcommand{\mZ}{{\mathbb Z}}

\newcommand{\bk}{{\mathbf k}}

\newcommand{\bx}{{\mathbf x}}

\newcommand{\bA}{\mathbf A}

\newcommand{\rk}{\textnormal{rk}}

\newcommand{\prk}{\textnormal{prk}}

\newcommand{\del}{\triangle}

\author{Amichai Lampert and Tamar Ziegler}
\thanks{The authors were supported by ERC grant ErgComNum 682150, and ISF grant 2112/20.}
\title{On rank in algebraic closure}
\begin{document}

\begin{abstract}
	 Let $ \bk $ be a field and $Q\in \bk[x_1, \ldots, x_s]$ a form (homogeneous polynomial) of degree $d>1.$  The $\bk$-Schmidt rank $\rk_\bk(Q)$ of $Q$ is the minimal $r$ such that $Q= \sum_{i=1}^r R_iS_i$ with $R_i, S_i \in \bk[x_1, \ldots, x_s]$ forms of degree $<d$. When $ \bk $ is algebraically closed and $ \textnormal{char}(\bk)$ doesn't divide $d,$ this rank is closely related to $ \textnormal{codim}_{\mA^s} (\nabla Q(x) = 0)$ - also known as the Birch rank of $ Q. $ When $ \bk $ is a number field,  a finite field or a function field, we give polynomial bounds for $ \rk_\bk(Q) $ in terms of $ \rk_{\bar \bk} (Q) $ where $ \bar \bk $ is the algebraic closure of $ \bk. $ Prior to this work no such bound (even ineffective) was known for $d>4$. This result has immediate consequences for counting integer points (when $ \bk $ is a number field) or prime points (when $ \bk = \mQ $) of the variety $ (Q=0) $ assuming $ \rk_\bk (Q) $ is large. 
\end{abstract}

	\maketitle 

\section{Introduction}

In \cite{SCH85}, Schmidt made the following definition:
 
 \begin{definition}
 	 Let $ \bk $ be a field and $Q\in \bk[x_1, \ldots, x_s]$ a form, i.e. homogeneous polynomial. The Schmidt rank $\rk_\bk(Q)$  is the smallest natural number $r$ such that $Q= \sum_{i=1}^r R_iS_i$ with $R_i, S_i\in \bk[x_1, \ldots, x_s]$ forms of positive degree. If $ Q $ is any degree $ d $ polynomial, we define $ \rk_\bk(Q) $ as the rank of its degree $ d $ part. 
 \end{definition}

Note that for $l\in \bk[x_1, \ldots, x_s]$ linear and non-constant, $\rk_\bk (l) =\infty.$
Schmidt introduced this quantity when $ \bk = \mQ $ and $ Q $ is homogeneous and proved that if $ \rk_\mQ(Q) $ is sufficiently large then a certain local-global rule applies for counting the number of integer points in the set $ \{x\in \mZ^s, |x| \le P: Q(x) = 0\}. $ He also showed that $ \rk_\mC(Q) $ is essentially equivalent to $\textnormal{codim}_{\mC^s} (\nabla Q(x)=0), $ known as the Birch rank of $ Q $ (we will define Birch rank more generally in the next section). Note that it follows immediately from the definition that
$ \rk_\mC(Q) \le \rk_\mQ(Q). $ The object of this paper is to prove a polynomial bound in the opposite direction, for $ \bk=\mQ $ and several other fields. For perfect fields $\bk,$ inequalities bounding $ \rk_{\bk}(Q) $ in terms of $ \rk_{\bar{\bk}}(Q) $ are known for $ d=2,3 $ \cite{D20} and for $ d=4 $ there is an ineffective bound when $ \textnormal{char}(\bk)\neq 2 $ \cite{KP-quart-21}. The bound we prove is quite interesting in its own right, but also has a number of immediate applications to Diophantine equations, since many results which use the Hardy-Littlewood circle method assume the Birch rank is large. Our result allows us to relax this assumption to large Schmidt rank over the base field. 

We describe some direct applications of our result in the next section - counting prime solutions to Diophantine equations, counting integer solutions to polynomial equations over number fields, and an analogue for function fields. 

\begin{definition}
	We call the following fields \emph{admissible}:  number fields, finite fields and finite separable extensions of $ \mF_q(t).$
\end{definition}

We prove the following:
 
\begin{theorem}[Schmidt rank and field extensions] \label{main-single}
	 Let $\bk$ be an admissible field, $ \bar{\bk} $ its algebraic closure, and $Q\in \bk[x_1, \ldots, x_s]$ a polynomial of degree $ d>1 $ such that $ \textnormal{char}(\bk) > d  $ or $ \textnormal{char}(\bk) = 0.  $ Then there exist constants $ A = A(\bk,d), B = B(\bk,d) $ such that 
	\[ 
	\rk_\bk(Q) \le A [\rk_{\bar{\bk}}(Q)+1]^B.
	\]
\end{theorem}

We will soon give explicit values for the constants $ A,B. $ Our result generalizes to collections of forms of the same degree, with an appropriate definition of rank.

\begin{definition}
	Let $ Q_1,\ldots,Q_n\in  \bk[x_1,\ldots,x_s] $ be homogeneous polynomials all of degree $ d>1. $ Their collective Schmidt rank is defined to be
	\[ 
	\rk_\bk(Q_1,\ldots,Q_n) : = \min_{0\neq a\in \bk^n} \rk_\bk \left( \sum_{i=1}^n a_i\cdot Q_i \right). 
	\]
	If $ Q_1,\ldots,Q_n $ are all of degree $ d $ but not necessarily homogeneous, we define $ \rk_\bk(Q_1,\ldots,Q_n)  $ in terms of their degree $ d $ parts.
\end{definition}

\begin{theorem}[Schmidt rank and field extensions]\label{main-multiple}
	Let $\bk$ be an admissible field, $ \bar{\bk} $ its algebraic closure, and $Q_1,\ldots,Q_n\in \bk[x_1, \ldots, x_s]$ a collection of polynomials of common degree $ d>1 $ such that $ \textnormal{char}(\bk) > d  $ or $ \textnormal{char}(\bk) = 0.  $ Then there exist constants $ A = A(\bk,d), B = B(\bk,d) $ such that if $ \bk $ is a perfect field we have 
	\[ 
	\rk_\bk(Q_1,\ldots,Q_n) \le A [n\cdot \rk_{\bar{\bk}}(Q_1,\ldots,Q_n)+1]^B,
	\]
	and if $ \bk $ is an imperfect field we have
	\[ 
	\rk_\bk(Q_1,\ldots,Q_n) \le A \left[ n^{1+1/d}\cdot (\rk_{\bar{\bk}}(Q_1,\ldots,Q_n)+1)\right] ^B.
	 \]
	 We can take $ A(\bk,d) , B(\bk,d) $ as follows:
		\begin{enumerate}
		\item $  A = 4^{d-1}d\binom{d}{\lfloor d/2 \rfloor}^d , \ B = d $ when $\bk = \mQ. $
		\item   $A  = 2^{2^{O(d^2)}},\   
    B =  2^{2^{O(d^2)}}$  when $ \bk $ is a finite field.
		\item $A = 2^{d-1}\binom{d}{\lfloor d/2 \rfloor}^d,\   
    B = d$ when $\bk$ is a finite field and $|\bk| \ge F(\rk_\bk (Q_1,\ldots,Q_n),n,d).$  
		\item  $ A = 2^{d-1}(d-1) \binom{d}{\lfloor d/2 \rfloor}^{2d}, \  
    B = 2d$ when $ \bk $ is a number field or a finite separable extension of $ \mF_q(t). $
	\end{enumerate}
\end{theorem}

\pagebreak

We mention a conjecture stated in \cite{AKZ21}

\begin{conjecture}
	Let $ d>1 $ and let $\bk$ be a field with $ char(\bk) > d $ or $ char(\bk) = 0 $
	There exists a constant $ C(d) $ such that for any polynomial $ Q\in \bk[x_1,\ldots,x_s] $ of degree $ d $ we have 
	\[ 
	\rk_\bk(Q) \le C\cdot \rk_{\bar \bk} (Q).
	\] 
\end{conjecture}

This conjecture is known for $d=2,3$ when $\bk$ is a perfect field  \cite{D20}.
\newline

The paper is organized as follows: In section 2 we give some number theoretic applications of our results. In section 3 we explain how to deduce them from an analogous theorem about multilinear polynomials. In section 4 we prove theorem \ref{main-multiple} for finite fields and for $ \mQ $ (the result for $ \mQ $ relies on the finite fields result), and in section 5 we prove it for number fields and for finite separable extensions of $\mF_q(t).$

\subsection{Acknowledgements}
The authors thank the anonymous referee for a careful reading and helpful suggestions.

\section{Applications}

In \cite{B62}, Birch defined a geometric notion of rank which is closely related to the singular locus of the variety defined by a collection of polynomials. In many applications of the Hardy-Littlewood circle method, it's necessary to assume that the Birch rank of the polynomials we are interested in is large. 

\begin{definition}
	Let $ Q\in \bk[x_1,\ldots,x_s] $ be a form of degree $ d. $ The Birch rank of $ Q $ is $ \rk_B(Q) = \textnormal{codim}_{\mathbb{A}^s} (\nabla Q(x) = 0). $ Given a collection $ Q_1,\ldots,Q_n\in \bk[x_1,\ldots,x_s] $ of forms of degree $ d, $ we define $ S(Q_1,\ldots,Q_n) $ to be the variety where the matrix of partial derivatives $ (\partial Q_i/\partial x_j) $ has rank $ < n. $ The Birch rank of the collection is $ \rk_B(Q_1,\ldots,Q_n) = \textnormal{codim}_{\mathbb{A}^s} S(Q_1,\ldots,Q_n). $ For a general collection of polynomials $ Q_1,\ldots,Q_n\in \bk[x_1,\ldots,x_s] $ of degree $ d, $ we define their Birch rank in terms of the homogeneous degree $ d $ parts. As a matter of convention, set $ \rk_B(\emptyset) = s. $
\end{definition}

It's not difficult to show that there is an inequality
\[ 
\rk_B(Q_1,\ldots,Q_n) \le 2\rk_\bk(Q_1,\ldots,Q_n). 
\]

Let us discuss for a moment what happens with quadratics. For a single quadratic $ Q $, there is an elementary inequality in the opposite direction
$
\rk_\bk(Q) \le \rk_B(Q).
$
If $ \bk $ is algebraically closed, we can take the union of the Birch singular loci of non-trivial linear combinations $ c_1Q_1+\ldots+c_nQ_n, $ which correspond to points of $ \mathbb P^{n-1}, $ and apply the above inequality to deduce:
\[ 
\rk_\bk(Q_1,\ldots,Q_n) \le \rk_B(Q_1,\ldots,Q_n)+n-1. 
\]
As a consequence, when $ \rk_\bk(Q_1,\ldots,Q_n)\ge 2n-1 $ we deduce that $ X = \{Q_1=\ldots=Q_n=0\} $ is a complete intersection of codimension $ n. $ For suppose $ Y\subset X $ has codimension $ <n. $ Since $ \rk_B(Q_1,\ldots,Q_n) \ge n, $ there exists a point $ y\in Y $ such that the Jacobian matrix has rank $ n $ at $ y. $ But then the tangent space at $ y $ has codimension $ n, $ contradicting the fact that the local codimension at $ y $ is $ <n. $

Kazhdan, Polishchuk and the first author \cite{KP21} extended results of Schmidt to show that a similar inequality holds also for higher degrees. Cohen and Moshkovitz \cite{COMO21} proved a similar result for a single multi-linear polynomial.

\begin{theorem}[Theorem 1.4 in \cite{KP21}]\label{rk-sing}
	Assume that $ d\ge 2, $ $ \bk $ is algebraically closed, and $ \textnormal{char}(\bk) $ doesn't divide $ d. $ Then we have
	\[ 
	\rk_\bk(Q_1,\ldots,Q_n) \le (d-1)\left[ \rk_B(Q_1,\ldots,Q_n)+n-1\right] .
	\]
	In particular, if $ \rk_\bk(Q_1,\ldots,Q_n) \ge (d-1)(2n-1) $ then
	$ (Q_1=\ldots=Q_n=0) $
	 is a complete intersection of codimension $ n.$
\end{theorem}

\begin{corollary}\label{rk-to-Birch}
	Let $ \bk $ be an admissible field, and $ Q_1,\ldots,Q_n \in \bk[x_1,\ldots,x_s] $ forms of degree $ d\ge 2, $ such that $ \textnormal{char}(\bk) > d $ or $ \textnormal{char}(\bk) = 0. $ Then there exist constants $ A' = A'(\bk,d), B' = B'(\bk,d) $ such that if $ \rk_\bk(Q_1,\ldots,Q_n) \ge A'(nr)^{B'} $ then $\rk_B(Q_1,\ldots,Q_n) \ge r. $ In particular, if $r\ge n$ then $(Q_1=\ldots=Q_n =0 )$ is a complete intersection of codimension $n.$
\end{corollary} 

\begin{proof}
    If $A',B'$ are taken sufficiently large then by theorem \ref{main-multiple} we obtain
    \[
        \rk_{\bar{\bk}} (Q_1,\ldots,Q_n) \ge (d-1)(r+n-1).
    \]
    Plugging this into theorem \ref{rk-sing} proves the corollary. 
\end{proof}

As we mentioned above, many applications of the Hardy-Littlewood circle method require that the Birch rank of the collection of forms be large. The above corollary which follows immediately from theorem \ref{main-multiple} allows us to replace this condition by requiring the Schmidt rank to be large. We now give some examples of this.

\subsection{Prime solutions}
Yamagishi \cite{Y18}, extending a result of Cook and Magyar \cite{CM13}, proved a local to global rule holds for counting prime solutions to systems of integral polynomial equations with sufficiently large Birch rank. Cook and Magyar conjectured that it's sufficient to assume that the Schmidt rank over $ \mQ $ is large. The conjecture was proved for polynomials of degree $ \le 3 $ in \cite{XY20}. As a consequence of our results, we can prove their conjecture for polynomials of any degree. Let $D \geq 1$, and let $\bar{Q} = (\bar{Q}_1, \ldots, \bar{Q}_D)$ be a system of polynomials in $\mathbb{Z}[x_1, \ldots, x_s]$, where
$\bar{Q}_d = ( Q_{d,1} , \ldots, Q_{d, t_d })$ is the subsystem of degree $d$ polynomials of $\bar{Q}$ $(1 \leq d \leq D)$. Let $ X = \{\bar{Q} = 0\} $ be the affine variety defined by $ \bar{Q}. $ Following Cook and Magyar, we use a different notion of rank for systems of linear polynomials. If $ l_1,\ldots,l_n $ are linear polynomials with degree one homogeneous parts $ L_1,\ldots,L_n $ we define $r (l_1,\ldots,l_n)$ to be the minimum number of non-zero coefficients
in a non-trivial linear combination
$$
\lambda_1 L_1 + \ldots + \lambda_n L_n,
$$
where $\boldsymbol{\lambda} = (\lambda_1, \ldots, \lambda_n) \in \mathbb{Q}^n \backslash \{ \mathbf{0} \}.$ Let $\Lambda$ be the von Mangoldt function,
Given $x = (x_1, \ldots , x_s)$, we let
\begin{equation*}
	\label{Lambda}
	\Lambda(x) = \Lambda(x_1) \ldots \Lambda(x_s).
\end{equation*}
Set
$$
\mathcal{M}_{\bar{Q}}(N) := \sum_{x \in [0, N]^s} \Lambda(x)  \  1_X(x).
$$
As an immediate consequence of combining corollary \ref{rk-to-Birch} with Yamagishi's result we obtain:

\begin{theorem}
	There exists a constant $ \chi = \chi(D,t_1,\ldots, t_D) $ such that if $ r(\bar{Q}_1) \ge \chi $ and $ \rk_\mQ(\bar{Q}_d) \ge \chi $ for each $ 2\le d\le D $ with $ t_d>0 $ 
	then there exist constants $\mathcal{C} = \mathcal{C}(\bar{Q})$, and
	$c>0$ such that
	$$
	\mathcal{M}_{\bar{Q} }(N) = \mathcal{C}(\bar{Q}) \ N^{s- \sum_{d = 1}^D d t_d} + O\left( \frac{
		N^{s-\sum_{d = 1}^D d t_d}
	}{(\log N)^c } \right).
	$$
	Furthermore, if the system of equations $ \bar{Q}=0 $ has a non-singular solution in $\mathbb{Z}_p^{\times}$, the units of $p$-adic integers, for each prime $p$, and the system of homgeneous equations given by the leading forms has a non-singular real point in $(0,1)^s$, then
	$$
	\mathcal{C}(\bar{Q}) > 0.
	$$
\end{theorem}

\subsection{Number fields}

Frei and Madritsch \cite{FM17}, extending an earlier result of Skinner \cite{SK97}, proved a local to global rule for counting integral solutions to a system of polynomial equations over a number field, assuming the Birch rank is sufficiently large. Let $K$ be a number field of degree $n$ over $\mQ$. Let $ \bar{Q} $ be a
system of polynomials in $s$ variables over the ring of integers $\mathcal{O}$
of $K$ with maximal degree $ D. $ We assume the polynomials
to be ordered by their degrees, that is, for each $d \in \{1, \ldots,
D\}$, we are given polynomials
\begin{equation*}
	\bar{Q}_d = (Q_{d,1}, \ldots, Q_{d,t_d}) \in \mathcal{O}_K[x_1, \ldots, x_s]
\end{equation*}
of degree $d$, where $t_d \geq 0$ with $t_D\geq1$. Fix an integral ideal $\eta$ of $\mathcal{O}$ and a $\mZ$-basis
$\omega_1, \ldots, \omega_n$ of $\eta$. We will also consider $\omega_1,
\ldots, \omega_n$ as an $\mR$-basis of $V := K \otimes_\mQ \mR$. By a
\emph{box $\mathcal{B}$ aligned to the basis}, we mean the set of all
$x = (x_1, \ldots, x_s) \in V^s$, where each $x_i$ has the form
$x_{i,1}\omega_1 + \cdots + x_{i,n}\omega_n$, such that the
coordinates $(x_{i,j})_{i,j}\in\mR^{ns}$ lie in a given box $B
\subseteq [-1,1]^{ns}$ with sides aligned to the coordinate axes of
$\mR^{ns}$. Given such a box $\mathcal{B}$, we define the counting function
\begin{equation*}
	N(P) := \left| \{x \in \eta^s\cap P \mathcal{B}: Q_{d,i}(x)=0 \textnormal{ for all }1\leq d \leq D, 1\leq i \leq t_d\} \right| .
\end{equation*}
Let
\begin{equation*}
	\Delta := \{ 1\leq d\leq D : t_d \geq 1\}.
\end{equation*}
Let $\mathcal{D}_0:=0$ and, for $1\leq d \leq D$,
\begin{equation*}
	\mathcal{D}_d := t_1 + 2 t_2 + \cdots + d t_d,
\end{equation*}
so that $\mathcal{D}:=\mathcal{D}_D$ is the sum of the degrees of all
the polynomials $Q_{d,i}$. 	Set
\begin{align*}
	u_d &:= \sum_{k=d}^D2^{k-1}(k-1)t_k \quad\textnormal{ for }\quad 1\leq d \leq D+1,\\
	s_0(d) &:= \mathcal{D}_d(2^{d-1}+u_{d+1}) + u_{d+1} + \sum_{j=d+1}^Du_jt_j\\
	s_0 &:= \max\{s_0(d): d\in\Delta\cup \{0\}\}.
\end{align*}

By combining corollary \ref{rk-to-Birch} with the result of Frei and Madritsch we obtain:
\begin{theorem}
	There exist constants $A(d),B(d)$ such that if for all $ d\in \Delta $ we have $\rk_K(\bar{Q}_d) > A(t_ds_0)^B$ then there is a positive
	$\delta$ with
	\begin{equation*}
		N(P) = \mathfrak{S}\mathfrak{J}\cdot P^{n(s-\mathcal{D})} + O(P^{n(s-\mathcal{D})-\delta})
	\end{equation*}
	for $P\to \infty$. Here, $\mathfrak{S}$ is the usual singular series
	and $\mathfrak{J}$ is the usual singular integral.
\end{theorem}

\subsection{Function fields}

For function fields, Lee proved a local to global rule for systems of polynomials with high Birch rank \cite{LEE11}. Suppose $ \bk = \mF_q(t) $ where $ q $ is a prime power and write $ \bA = \mF_q[t]. $ Let $ \bar{Q} = (Q_1,\ldots,Q_n)\in \bA[x_1,\ldots,x_s] $ be forms of degree $ d. $ Let $ |\cdot| $ be the absolute value on $ \bk $ given by $ |f/g|= q^{\deg f-\deg g}. $ For $ x\in\bk^s $ we define $\|x\|_\infty = \max_{i\in[s]} |x_i|. $ For $ P>0 $ set
\[ 
N(P) = \# \{x\in \bA^s: \|x\|_\infty < q^P, Q_i(x) = 0\ \forall i\in[n]\}.
\]

Corollary \ref{rk-to-Birch}, together with Lee's result yields:

\begin{theorem}
	There exist constants $A(d),B(d))$ such that the following holds: Assume $ \textnormal{char}(\bk) >d, $ $ \rk(\bar{Q}) > An^B, $ and that $ \{\bar{Q} = 0\} $ contains a non-singular point in the completion of $ \bk $ at every place. Then we have 
	\[
	N(P) = \mu q^{P(s-nd)} + O(q^{P(s-nd)-\delta})
	\]
	where $ \mu,\delta >0. $
\end{theorem}

\section{Reduction to multilinear polynomials and universality}

In this section we reduce theorem \ref{main-multiple} to a statement about multilinear forms. Let $ V $ be a finite dimensional vector space over $\bk$ and $ Q\in \bk[V] $ a form of degree $ d. $ Suppose either $\textnormal{char}(\bk) = 0$ or $ \textnormal{char}(\bk) > d. $ Then there is an associated multilinear form $ \tilde{Q}:V^d\to\bk$ given by $ \tilde{Q}(\bx_1,\ldots,\bx_d) = \del_{\bx_1}\ldots\del_{\bx_d}Q $ which satisfies $ Q(x) = \frac{1}{d!} \tilde{Q}(x,\ldots,x). $  For multilinear forms such as this there is a finer notion of rank.

\begin{definition}
	Let $\bk$ be any field, $ V_1,\ldots,V_d $ finite dimensional vector spaces over $ k, $ and $ P:V_1\times\ldots\times V_d\to\bk$ a multilinear form. We say that $ P $ has \emph{partition rank} $ 1 $ if there exists a partition $ [d] = I\sqcup J $ with $ I,J\neq\emptyset $ and multilinear forms $ R:\prod_{i\in I} V_i\to k, S: \prod_{j\in J} V_j\to k$ such that
	\[ 
	P(\bx) = R((\bx_i)_{i\in I})\cdot S((\bx_j)_{j\in J}).
	\]
	In general the partition rank $ \prk_\bk(P) $ is the smallest natural number $ r $ such that $ P = P_1+\ldots+P_r $ and $ P_i $ has partition rank $ 1 $ for all $ 1\le i\le r. $ 
\end{definition}

The next claim states that the partition rank of $ \tilde{Q} $ is essentially equivalent to the rank of $ Q. $

\begin{claim} \label{rank-eq}
Let $ V $ be a finite dimensional vector space over $\bk$ and $ Q\in \bk[V] $ a form of degree $ d. $ Suppose either $\textnormal{char}(\bk) = 0$ or $ \textnormal{char}(\bk) > d. $ then
	\[ 
	\rk_\bk(Q) \le \prk_\bk (\tilde{Q}) \le \binom{d}{\lfloor d/2 \rfloor}\rk_\bk(Q).
	\]
\end{claim} 

\begin{proof}
	For the left inequality, suppose $ \tilde{Q}(\bx) = \sum_{i=1}^r R(\bx)\cdot S(\bx) $ and plug in 
	\[ 
	Q(x) = \frac{1}{d!} \tilde{Q}(x,\ldots,x) = \frac{1}{d!} \sum_{i=1}^r R_i(x,\ldots,x)\cdot S_i(x,\ldots,x).
	\]
	For the right inequality, suppose $ Q(x) = \sum_{i=1}^r R_i(x)S_i(x) $ and take derivatives
	\[  
	\tilde{Q}(\bx_1,\ldots,\bx_d) = \sum_{i=1}^r \sum_{J\subset [d], |J|=\deg 
		R_i}  \tilde{R_i}((\bx_j)_{j\in J}) \tilde{S_i}((\bx_\bk)_{k\in [d]\setminus J}).
	\]
\end{proof} 

Partition rank is also defined for collections of multi-linear forms.

\begin{definition}
	 Let $P_1,\ldots,P_n:V_1\times\ldots\times V_d\to\bk$ be multilinear forms. Their partition rank is
	  \[
	  \prk_\bk(P_1,\ldots,P_n) = \min_{0\neq a\in \bk^n} \prk_\bk \left( \sum_{i=1}^n a_i\cdot P_i \right).
	  \]
\end{definition}

By claim \ref{rank-eq}, we can deduce theorem \ref{main-multiple} from a corresponding theorem for multi-linear forms. 

\begin{theorem}[theorem \ref{main-multiple} for multi-linear forms]\label{main-ml}
	Let $\bk$ be an admissible field, $ V_1,\ldots,V_d $ finite dimensional vector spaces over $ \bk, $ and \linebreak $P_1,\ldots,P_n:V_1\times\ldots\times V_d\to\bk$ multi-linear forms. There exist constants $  \tilde{A}(\bk,d) , \tilde{B}(\bk,d) $ such that if $ \bk $ is a perfect field we have 
	\[ 
	\prk_\bk(P_1,\ldots,P_n) \le \tilde{A} [n\cdot \prk_{\bar{\bk}}(P_1,\ldots,P_n)+1]^{\tilde{B}},
	\]
	and if $ \bk $ is an imperfect field we have
	\[ 
	\prk_\bk(P_1,\ldots,P_n) \le \tilde{A} [n^{1+1/d}\cdot (\prk_{\bar{\bk}}(P_1,\ldots,P_n)+1)]^{\tilde{B}}.
	\]
\end{theorem}

\begin{proof}[Proof of theorem \ref{main-multiple} assuming theorem \ref{main-ml}]
    We will show that theorem \ref{main-multiple} holds with $ A = \tilde{A}\binom{d}{\lfloor d/2 \rfloor}^{\tilde{B}}$ and $ B = \tilde{B}. $  Let $ \bk $ be an admissible field and let $ Q_1,\ldots,Q_n $ be degree $ d $ polynomials with $ \textnormal{char}(\bk) > d $ or $ \textnormal{char}(\bk) = 0. $ Let $ \tilde{Q_1},\ldots,\tilde{Q_n} $ be the associated multi-linear forms. If $ \bk $ is perfect then by claim \ref{rank-eq} and theorem \ref{main-ml} 
	\begin{multline*}
	\rk_\bk(Q_1,\ldots,Q_n) \le \prk_\bk(\tilde{Q_1},\ldots,\tilde{Q_n}) \le \tilde{A} [n\cdot \prk_{\bar{\bk}}(\tilde{Q_1},\ldots,\tilde{Q_n})+1]^{\tilde{B}} \\
	\le \tilde{A} \left[ n\cdot \binom{d}{\lfloor d/2 \rfloor} \rk_{\bar{\bk}}(Q_1,\ldots,Q_n)+1 \right] ^{\tilde{B}} \le \tilde{A}\binom{d}{\lfloor d/2 \rfloor}^{\tilde{B}} \left[ n\cdot  \rk_{\bar{\bk}}(Q_1,\ldots,Q_n)+1 \right] ^{\tilde{B}}.
	\end{multline*}
	If $ \bk $ is imperfect then again by claim \ref{rank-eq} and theorem \ref{main-ml} 
	\begin{multline*}
		\rk_\bk(Q_1,\ldots,Q_n) \le \prk_\bk(\tilde{Q_1},\ldots,\tilde{Q_n}) \le \tilde{A} [n^{1+1/d}\cdot (\prk_{\bar{\bk}}(\tilde{Q_1},\ldots,\tilde{Q_n})+1)]^{\tilde{B}} \\
		\le \tilde{A} [n^{1+1/d}\cdot (\binom{d}{\lfloor d/2 \rfloor}\rk_{\bar{\bk}}(Q_1,\ldots,Q_n)+1)]^{\tilde{B}} \le \tilde{A}\binom{d}{\lfloor d/2 \rfloor}^{\tilde{B}} \left[ n^{1+1/d}\cdot (\rk_{\bar{\bk}}(Q_1,\ldots,Q_n)+1) \right] ^{\tilde{B}}.
	\end{multline*}
\end{proof}

We prove theorem \ref{main-ml} for all of the above fields (except $ \mQ $) via a universality result, which may be of independent interest. This generalizes a universality result of Kazhdan and Ziegler appearing in \cite{KZ21}. 

\begin{theorem}[Universality]\label{universal}
	Let $\bk$ be an admissible field and $ R_1,\ldots,R_n:(\bk^t)^d\to\bk$ multilinear forms of degree $ d. $ There exist constants $ C(\bk,d) , D(\bk,d) $ such that whenever $ P_1,\ldots,P_n:V_1\times\ldots\times V_d\to\bk$ are multilinear forms with $ \prk_\bk(P_1,\ldots,P_n) > C(nt^d)^D, $ there exist linear maps $T_i: \bk^t \to V_i$ such that $P_j\circ (T_1\times\ldots\times T_d) = R_j. $ 
\end{theorem}

We now explain how theorem \ref{main-ml} follows from theorem \ref{universal}. Theorem \ref{universal} will be proved in the following sections.

\begin{proof}[Proof of theorem \ref{main-ml} assuming theorem \ref{universal}]
	
 	We will show that theorem \ref{main-ml} holds with $\tilde{A} = C, \tilde{B} = dD.$ Supposing
 	$$\prk_\bk(P_1,\ldots,P_n) > C(n^{1+1/d} \bar{r})^{dD},$$
 	we will show that $ \prk_{\bar{\bk}} (P_1,\ldots,P_n) \ge \bar{r}. $ By theorem \ref{universal} with $ t=n\bar{r}, $ there exist linear maps  $ T_i:\bk^t\to V_i $ such that
	\[ 
	P_j\circ T = \sum_{k=(j-1)\bar{r}+1}^{j\bar{r}} \prod_{i=1}^d \bx_i (k) =: D_j.
	\]
	Clearly $ \prk_{\bar \bk} (P_1,\ldots,P_n) \ge \prk_{\bar \bk} (D_1,\ldots,D_n). $ 
	To show that the right hand side equals $ \bar{r}, $ we need the following simple claim:
	
	\begin{claim}
		Let $ \bk $ be a field and $ P:V_1\times\ldots\times V_d \to \bk $ a multilinear form. The variety 
		\[ 
		Z_P := \{(v_2,\ldots,v_d): P(\cdot,v_2,\ldots,v_d) \equiv 0\}
		 \]
		satisfies
		\[ 
		\textnormal{codim}_{V_2\times\ldots\times V_d} Z_P \le \prk_\bk(P).
		 \]    
	\end{claim} 

	\begin{proof}
		Suppose $ P = \sum_{k=1}^r R_k((\bx_i)_{i\in I_k})\cdot S_k((\bx_j)_{j\in [d]\setminus I_k}) $ and assume without loss of generality that $ 1\notin I_k $ for all $ k. $ Then the variety $ \{R_1 = \ldots = R_r = 0\} $ has codimension $ \le r $ and is contained in $ Z_P. $
	\end{proof}

	\begin{lemma}
		$ \prk_{\bar \bk} (D_1,\ldots,D_n) = \bar{r}. $
	\end{lemma}
	
	\begin{proof}
		Obviously the partition rank is $ \le \bar{r}. $ On the other hand, let $ D = \sum_{j=1}^n a_jD_j $ be a non-trivial linear combination, and assume without loss of generality that $ a_1,\ldots,a_m \neq 0 $ and $ a_{m+1} = \ldots = a_n = 0 $ for some $ m\ge 1. $ We can decompose $ Z_D $ explicitly,
		\[ 
		Z_D = \{ \prod_{i=2}^d \bx_i (k) = 0 : \forall k\in [m\bar{r}] \} = \bigcup_{\sigma:[m\bar{r}]\to[2,d]} \{ \bx_{\sigma(j)} (j) = 0 : \forall j\in [m\bar{r}] \}.
		 \] 
		 $ Z_D $ is a finite union of linear subspaces of codimension $ m\bar{r}, $ so \linebreak
   $ \textnormal{codim}_{V_2\times\ldots\times V_d} Z_D = m\bar{r} \ge \bar{r}. $ By the previous claim, we get \linebreak $ \prk_{\bar \bk} (D_1,\ldots,D_n) \ge \bar{r}. $
	\end{proof}

This completes the proof of theorem \ref{main-ml} for all fields in the case $ n=1 $ and for imperfect fields in the case of general $ n. $ To obtain our improved bounds for perfect fields and general $ n $ we use the following argument which we learned from A. Polishchuk. 

\begin{proof}[Proof of theorem \ref{main-ml}  for perfect fields assuming the case $ n=1 $] \label{perfect-fields}
	Let
	\[ 
	\bar{r} = \prk_{\bar{\bk}}(P_1,\ldots,P_n) , \mathcal{P} = \textnormal{sp}_{\bar{\bk}} (P_1,\ldots,P_n) 
	 \]
	Choose $ 0\neq P\in \mathcal{P} $ with $ \rk_{\bar{\bk}} (P) = \bar{r}. $ Let $ \mathcal{P}'\subset \mathcal{P} $ be the $ {\bar{\bk}} $-subspace spanned by $ P $ and its Galois conjugates. By Galois descent, $ \mathcal{P}' = \textnormal{sp}_{\bar{\bk}} (P'_1,\ldots,P'_m) $ for some polynomials $ P'_1,\ldots,P'_m\in \textnormal{sp}_\bk (P_1,\ldots,P_n). $ Let $ \mathcal{P'}(\bk) := \textnormal{sp}_\bk (P'_1,\ldots,P'_m)  $ and choose $ 0\neq P'\in \mathcal{P'}(\bk). $ Since $ \dim \mathcal{P'}\le n, $ there are Galois conjugates $ \sigma_1\cdot P,\ldots,\sigma_n\cdot P $ such that $ P'\in \textnormal{sp}_{\bar{\bk}} (\sigma_1\cdot P,\ldots,\sigma_n\cdot P). $ By our assumption on $ P $ this means $ \rk_{\bar{\bk}}(P') \le n\bar{r}, $ which by theorem \ref{main-ml} implies
	\[ 
	\rk_\bk (P') \le \tilde{A} (n\bar{r}+1)^{\tilde{B}}.
	 \]
\end{proof}

\end{proof}

We now begin the proof of theorem \ref{universal}, which will be completed in subsequent sections. Assume without loss of generality that $ V_i = \bk^{s_i}. $ Then linear maps $ T_i : \bk^t\to V_i $ are of the form $ T_i(x) = A_i\cdot x $ where $ A_i\in \mathcal{M}_{s_i\times t}(\bk). $ For $ j_1,\ldots,j_d,l $ the coefficient of $ y^1_{j_1}\cdot\ldots\cdot y^d_{j_d} $ in $ P_l\circ (T_1\times\ldots\times T_d) $ is given by a multilinear form
\[ 
P^l_{j_1,\ldots,j_d}: \mathcal{M}_{s_1\times t}(\bk)\times\ldots\times\mathcal{M}_{s_d\times t}(\bk) \to \bk, 
\]
namely $ P^l_{j_1,\ldots,j_d}(A_1,\ldots,A_d) = \left( P_l\circ (A_1\times\ldots\times A_d) \right) (e_{j_1},\ldots,e_{j_d}). $ 

\begin{claim}\label{rank-universal}
	The collection of multi-linear forms thus obtained satisfies
	\[ 
	\prk_\bk \left( (P^l_{j_1,\ldots,j_d})_{l,j_1,\ldots,j_d} \right) = \prk_\bk (P_1,\ldots,P_n).
	 \]
\end{claim}

\begin{proof}
	For a multilinear form $ P=\sum_{k_1,\ldots,k_d} a_{k_1,\ldots,k_d} \prod_{i\in [d]} \bx_i(k_i) $ and indices $ j_1,\ldots,j_d $
	\[ 
	P_{j_1,\ldots,j_d}(A_1,\ldots,A_d) = \sum_{k_1,\ldots,k_d} a_{k_1,\ldots,k_d} \prod_{i\in [d]} A_i(k_i,j_i). 
	 \]
	Therefore, $ P_{j_1,\ldots,j_d}(A_1,\ldots, A_d) $ is just $ P $ with relabeled variables, meaning they have the same rank. So for any $ c_1,\ldots,c_n $ and fixed $ j_1,\ldots,j_d $ we have
	\[
	\prk_\bk (\sum_l c_lP^l_{j_1,\ldots,j_d}) = \prk_\bk (\sum_l c_lP_l).
	\]
	Now let $ P' = \sum_{l,j_1,\ldots,j_d } c^l_{j_1,\ldots,j_d} P^l_{j_1,\ldots,j_d} $ be a non-trivial linear combination with, say, $ c^l_{1,\ldots,1} \neq 0 $ for some $l.$ Restricting to the subspace
    \[
    U = \{A_i(k,j) = 0: \forall i,k\ \forall j\neq 1 \}
    \]
 gives $ P'\restriction_U = \sum_l  c^l_{1,\ldots,1} P^l_{1,\ldots,1} $ so
 \[
 \prk_\bk (P') \ge \prk_\bk (\sum_l  c^l_{1,\ldots,1} P^l_{1,\ldots,1}) = \prk_\bk (\sum_l  c^l_{1,\ldots,1} P_l) \ge \prk_\bk (P_1,\ldots,P_n). 
 \]
 
\end{proof}

The desired equalities $ P_l\circ (A_1\times\ldots\times A_d) = R_l $ are in fact $ nt^d $ equations of the form
\begin{equation}\label{universal-eq}
	  P^l_{j_1,\ldots,j_d}(A_1,\ldots, A_d) = r^l_{j_1,\ldots,j_d}
\end{equation}
 and we have just proved that the collection of multi-linear forms appearing has partition rank $ \prk_\bk (P_1,\ldots,P_n). $  Our goal in the next sections is to show that such a system of equations has a $ \bk $-solution.

\section{Proof of results for finite fields and for $ \mQ $}

\subsection{Proof for finite fields:}
 We start by defining bias.

\begin{definition}
	Let $ V $ be a finite dimensional vector space over a finite field $\bk$ and $ f:V\to\bk$ some function. The bias of $ f $ is 
	\[ 
	\textnormal{bias}(f) = \left| \mE_{x\in V} \chi(f(x)) \right| 
	 \]
	 where $ \chi:k\to\mC $ is some non-trivial character, and $\mE_{x\in X}$ denotes $\frac{1}{|X|}\sum_{x\in X}$ for a finite set $X.$
\end{definition}

The bias of $ f $ depends also on the choice of character $ \chi $ but this won't matter in what follows. The key to proving theorem \ref{universal} is the following theorem relating rank to bias for multilinear polynomials. The first part is due to Milicevic \cite{M19}, with a similar result proved independently by Janzer \cite{J20}, and the second to Cohen-Moshkovitz \cite{COMO21}.

\begin{theorem}\label{bias-rank}
	Let $\bk$ be a finite field and $ P:V_1\times\ldots\times V_d\to\bk$ a multilinear form. Assume either of the following conditions holds for some $ r \ge 1 $:
	\begin{enumerate}
		\item $\prk_\bk(P) \ge \alpha r^\beta,$ for constants $ \alpha(d) = 2^{d^{2^{O(d^2)}}} , \beta(d) = 2^{2^{O(d^2)}}. $ 
		\item $ |\bk| \ge F(r,d) $ and $ \prk_\bk(P) \ge 2^{d-1}r. $	
	\end{enumerate}
	Then $ \textnormal{bias}(P) \le |\bk|^{-r}. $
\end{theorem}

The proof of theorem \ref{universal} for finite fields, and hence of theorems \ref{main-ml} and \ref{main-multiple}, now follows by a routine Fourier-analytic argument.

\begin{proof}
We will show that theorem \ref{universal} holds for finite fields with \linebreak $C = 2^{d^{2^{O(d^2)}}}, D = 2^{2^{O(d^2)}}$ or with $C=2^{d-1}, D=1$ under the additional assumption that $|\bk| \ge F(nt^d,d).$
Recall that we want to show the existence of a $ \bk$-solution to the system of $ nt^d $ multilinear equations (\ref{universal-eq})
\[ 
P^l_{j_1,\ldots,j_d}(A_1,\ldots, A_d) = r^l_{j_1,\ldots,j_d}.
 \]
We showed in claim \ref{rank-universal} that 
\[ 
\prk_\bk \left( P^l_{j_1,\ldots,j_d} \right)_{l,j_1,\ldots,j_d} = \prk_\bk(P_1,\ldots,P_n).  
 \]
Applying theorem $ \ref{bias-rank}, $ we get that any nontrivial linear combination satisfies 
\[ 
\textnormal{bias} \left(  \sum_{l,j_1,\ldots,j_d}\alpha^l_{j_1,\ldots,j_d} P^l_{j_1,\ldots,j_d} \right)  \le |\bk|^{-nt^d}.
 \] 
We can therefore estimate the number of solutions
	\begin{align*}
		&\mE_{A_1,\ldots,A_d} \prod_{l,j_1,\ldots,j_d} 1_{P^l_{j_1,\ldots,j_d} (A_1,\ldots,A_d) = r^l_{j_1,\ldots,j_d}} =\\
		 &\mE_{\alpha^l_{j_1,\ldots,j_d}} \mE_{A_1,\ldots,A_d} \chi (\sum_{l,j_1,\ldots,j_d}\alpha^l_{j_1,\ldots,j_d} (P^l_{j_1,\ldots,j_d} (A_1,\ldots,A_d) - r^l_{j_1,\ldots,j_d})) = \\
		 &|\bk|^{-nt^d} \left( 1+ \sum_{\alpha^l_{j_1,\ldots,j_d} \neq 0} \mE_{A_1,\ldots,A_d} \chi (\sum_{l,j_1,\ldots,j_d}\alpha^l_{j_1,\ldots,j_d} (P^l_{j_1,\ldots,j_d} (A_1,\ldots,A_d) - r^l_{j_1,\ldots,j_d}))\right) \ge \\
		 & |\bk|^{-nt^d} \left( 1 - \sum_{\alpha^l_{j_1,\ldots,j_d}\neq 0} \textnormal{bias}( \sum_{l,j_1,\ldots,j_d}\alpha^l_{j_1,\ldots,j_d} P^l_{j_1,\ldots,j_d} ) \right) > 0.
	\end{align*}
 As we saw in the previous section, this proves theorem \ref{main-ml} with the constants
 \begin{enumerate}
     \item $ \tilde{A} = C = 2^{d^{2^{O(d^2)}}},\ \tilde{B} = dD = 2^{2^{O(d^2)}}.$
     \item $ \tilde{A} = C = 2^{d-1},\ \tilde{B} = dD = d$ if $|\bk| \ge F(\prk_\bk (P_1,\ldots,P_n),n,d),$
 \end{enumerate}
and theorem \ref{main-multiple} with the constants
\begin{enumerate}
    \item $A = \tilde{A} \binom{d}{\lfloor d/2 \rfloor}^{\tilde{B}} = 2^{2^{O(d^2)}},\   
    B = \Tilde{B} = 2^{2^{O(d^2)}}.$
    \item $A = \tilde{A} \binom{d}{\lfloor d/2 \rfloor}^{\tilde{B}} = 2^{d-1}\binom{d}{\lfloor d/2 \rfloor}^d,\   
    B = \Tilde{B} = d$ if $|\bk| \ge F(\rk_\bk (Q_1,\ldots,Q_n),n,d).$
\end{enumerate}
\end{proof}

\subsection{Proof for $\mQ$:} Now we turn to proving theorem \ref{main-ml} for the rationals. We prove it in the case $ n=1. $ The case of general $ n $ follows by the argument given for perfect fields in the previous section. Let $ P: \mQ^{s_1}\times\ldots\times \mQ^{s_d}\to \mQ $ be a multi-linear form, and write $ s=s_1+\ldots+s_d. $ Assume without loss of generality that $ P $ has integer coefficients. We will show that for all but finitely many primes $p,$ the reduction $P\mod p$ has partition rank which is bounded from below in terms of $\prk_\mQ(P).$ 

\begin{proposition}\label{rk-to-finite}
    Let $ P: \mQ^{s_1}\times\ldots\times \mQ^{s_d}\to \mQ $ be a multi-linear form with integer coefficients. Then for all sufficiently large primes $p$ we have
    \[
        \prk_\mQ(P) \le 2^{d-1} d\cdot  \prk_{\mF_p}(P\mod p). 
    \]
\end{proposition}

We now complete the proof of theorem \ref{main-ml} and theorem \ref{main-multiple} for $\mQ$ assuming proposition \ref{rk-to-finite}.

\begin{proof} 
    Suppose $\prk_\mQ (P) \ge 4^{d-1}d\cdot \bar{r}^d.$ By proposition \ref{rk-to-finite} we get that for sufficiently large primes $p,$ we have $\prk_{\mF_p}(P\mod p) \ge 2^{d-1}\cdot \bar{r}^d.$ By theorem \ref{main-ml} for sufficiently large finite fields, we conclude that for $p$ sufficiently large we have $ \prk_{\bar{\mF_p}}(P \mod p) \ge \bar{r}. $ By a model-theoretic result (corollary 2.2.10 in \cite{Mod}), we deduce that $\prk_{\bar{\mQ}}(P) \ge \bar{r}.$ This completes the proof of theorem \ref{main-ml} with
    \[
    \tilde{A} = 4^{d-1}d, \tilde{B} = d.
    \]
    By the argument given in the last section this proves theorem \ref{main-multiple} with 
    \[
    A = \tilde{A} \binom{d}{\lfloor d/2 \rfloor}^{\tilde{B}} = 4^{d-1}d\binom{d}{\lfloor d/2 \rfloor}^d , \ B = \tilde{B} =d.
    \]
\end{proof}

In order to prove proposition \ref{rk-to-finite}, we will need a simple scaling result for solutions to systems of multilinear equations. 

\begin{lemma}\label{scaling}
    Let $\bar Q = Q_1,\ldots,Q_n:\mZ^{s_1}\times\ldots\times \mZ^{s_d}\to G$ be a collection of multilinear maps where $G$ is some abelian group. For a positive integer $R$ let 
    \begin{align*}
        N_R(\bar Q) &= |\{x\in \mZ^s: Q_i(x) = 0\ \forall i\in[n], 0\le x_j< R\  \forall j\in[s]\}| \\
        N'_R(\bar Q) &= |\{x\in \mZ^s: Q_i(x) = 0\ \forall i\in[n], -R< x_j< R\  \forall j\in[s]\}|.
    \end{align*}
    Then for any positive integer $L$ we have
    \[
        N_{LR}(\bar Q) \le L^s N'_R(\bar Q),
    \]
    where $s=s_1+\ldots+s_d.$
\end{lemma}

\begin{proof} 
    By induction on $d.$ For the base case $d=1,$ first note that for any $c_i\in G$ the number of solutions to $Q_i(x) = c_i$ in $[0,R)^s$ is bounded above by $N'_R(\bar Q).$ This follows by substracting a single solution from every one of the others. Writing $\bar{R} = (R,\ldots,R),$ we have
    \begin{align*}
        N_{LR}(\bar Q) &= \sum_{0\le x_i<LR} 1_{Q_i(x) = 0} = \sum_{\omega\in\{0,1,\ldots,L-1\}^s} \sum_{0\le x_i<R} 1_{Q_i(x+\omega\cdot \bar R) = 0} \\
        &= \sum_{\omega\in\{0,1,\ldots,L-1\}^s} \sum_{0\le x_i<R} 1_{Q_i(x) = -Q_i(\omega\cdot \bar R)} \le  \sum_{\omega\in\{0,1,\ldots,L-1\}^s} N'_R(\bar Q) = L^s N'_R(\bar Q).
    \end{align*}
    For general $d,$ we have
    \begin{align*}
         N_{LR}(\bar Q) &= \sum_{0\le \bx_1(i)<LR} N_{LR} (\bar Q(\bx_1,\cdot)) \le \sum_{0\le \bx_1(i)<LR} L^{s_2+\ldots+s_d}N'_R (\bar Q(\bx_1,\cdot)) \\
         &= L^{s_2+\ldots+s_d} \sum_{\|\bx_2\|_\infty,\cdots,\|\bx_d\|_\infty<R}  N_{LR}(\bar Q(\cdot,\bx_2,\cdots,\bx_d)) \\
         &\le L^s \sum_{\|\bx_2\|_\infty,\cdots,\|\bx_d\|_\infty<R}  N'_R(\bar Q(\cdot,\bx_2,\cdots,\bx_d)) = L^s N'_R(\bar Q),
    \end{align*}
   where the first inequality follows from the inductive hypothesis and the second one follows from the base case $d=1.$
\end{proof}

\begin{proof}[Proof of proposition \ref{rk-to-finite}]
    Let $s =s_1+\ldots+s_{d-1}.$ Suppose $p$ is a prime and \linebreak
    $\prk_{\mF_p} (P\mod p) = r.$ By the inequality $ \textnormal{bias}(P \mod p) \ge p^{-r}, $ proved in \cite{KZ17}, we have
    \[
        |\{\bx\in \mZ^s:P(\bx_1,\ldots,\bx_{d-1},\cdot) = 0 \mod p\}|\ge p^{s-r}.
    \]
    Given $0<\eta<1 ,$ we can apply lemma \ref{scaling} to the collection of multilinear maps
    \[
    \bar Q = Q_1,\ldots,Q_{s_d}: \mZ^{s_1}\times\ldots\times\mZ^{s_{d-1}}\to\mF_p,
    \]
    where
    \[
        Q_i(\bx) = P(\bx_1,\ldots,\bx_{d-1},e_i) \mod p,
    \]
     to get
     \[
         N_p(\bar Q) \le \lceil p^{1-\eta} \rceil ^s N'_{\lceil p^\eta \rceil}(\bar Q) \ll_s p^{s(1-\eta)}N'_{\lceil p^\eta \rceil}(\bar Q),
     \]
     so
     \[
     N'_{\lceil p^\eta \rceil}(\bar Q) \gg_s p^{\eta (s-r/\eta)} \gg_{s,\eta,r} \lceil p^\eta \rceil^{s-r/\eta}.
     \]
    For $\eta = \frac{1}{d}$ and $\|\bx\|_\infty < \lceil p^\eta \rceil$ we have $|P(\bx,e_i)| \ll p^{\frac{d-1}{d}}$ so for sufficiently large $p,$
    \[
        P(\bx,e_i) = 0\mod p \implies P(\bx,e_i) = 0.
    \]
    Therefore, 
    \[
    |\{\bx\in\mZ^s:P(\bx,\cdot) = 0, \|\bx\|_\infty < \lceil p^{1/d}\rceil \}| \gg_{s,r} \lceil p^{1/d} \rceil^{s-dr}. 
    \]
    By corollary \ref{points-to-rank}, which will be proved in the next section, this implies \linebreak
    $\prk_\mQ (P) \le 2^{d-1} d r.$
    
\end{proof}

\section{Proof for number fields and finite separable extensions of $\mF_q(t)$}

The obstruction to the solution to a system of equations given by multi-linear forms is the Birch singular locus.
\begin{definition}
	Let $ P:V_1\times\ldots\times V_d\to\bk $ be a multilinear form of degree $ d>1. $ The Birch singular locus is the subvariety
	\[ 
	V_1\times\ldots\times V_{d-1} \supset Z_P = \{(\bx_1,\ldots,\bx_{d-1}): P(\bx_1,\ldots,\bx_{d-1},\cdot) =0\}. 
	\]
	For a collection $ \mathcal{P} = (P_1,\ldots,P_n) $ the Birch singular locus is
	\[ 
	Z_{\mathcal{P}} = \{(\bx_1,\ldots,\bx_{d-1}): P_1(\bx_1,\ldots,\bx_{d-1},\cdot),\ldots,P_n(\bx_1,\ldots,\bx_{d-1},\cdot) \textnormal{ are linearly dependent}\}.
	\] 
\end{definition}

Kazhdan, Polishchuk and the first author \cite{KP21} extended a result of Schmidt \cite{SCH85} and showed a relationship between the Birch singular locus and the Schmidt rank for a multilinear form $ P $ over a general field $ \bk.$ 

\begin{theorem}\label{rank-sing}
    Let $ g_\bk(P) $ denote the codimension in $ V_1\times\ldots\times V_{d-1} $ of the Zariski closure of $ Z_P(\bk). $ Then $ \prk_\bk (P) \le 2^{d-1}\cdot g_\bk(P).$
\end{theorem}

By adapting an argument that Schmidt gave for $ \mQ $ in \cite{SCH85}, this extends to collections of multilinear forms over some fields. 

\begin{proposition}\label{rank-sing-families}
	Let $ \bk $ be a number field or a finite separable extension of $ \mF_q(t), $ and
	\[ 
	\mathcal{P} = (P_1,\ldots,P_n):\bk^{s_1}\times\ldots\times \bk^{s_d}\to\bk
	\]
	a collection of multilinear forms. Suppose for some $c_1,\ldots,c_n\in\bk$ the system of equations $(P_i = c_i)$ has no solution in $\bk^{s_1}\times\ldots\times \bk^{s_d}.$ Then 
 $$ \prk_\bk(\mathcal{P}) \le 2^{d-1}(d-1)n^2.$$ 
\end{proposition}

To prove proposition \ref{rank-sing-families} we need an effective version of Noether normalization.

\begin{lemma}[Noether normalization]\label{normal}
	Let $ \bk $ be an algebraically closed field and let $ V\subset \bk^s $ be an algebraic set of codimension $ t $ defined by $ m $ homogenenous equations of degree $ \le n. $ Look at matrices $ M\in\mathcal{M}_{(s-t)\times s}(\bk). $ There exists a constant $ C(s,m,n) $ (independent of $ V$!) such that for any set $ U\subset \bk $ of size $|U|\ge C$ we can find a matrix $ M $ with entries in $ U $ for which the map \linebreak 
	$ M:V\to\bk^{s-t} $ has finite fibers of size bounded by $C.$
\end{lemma}

\begin{proof}[Proof of lemma]
	Just follow the proof of linear Noether normalization, for example in \cite{Eis13}. Note that we need to choose the entries of $ M $ to avoid solutions of certain polynomial equations of degree $ \le n. $
\end{proof}

This lemma gives us a useful corollary when we have a pseudo-norm at our disposal. It will be convenient to use a slightly non-standard definition.

\begin{definition}
	Let $ \bA $ be a ring. We call a function $ \phi:\bA\to \mZ^{\ge 0} $ a \emph{pseudo-norm} if it satisfies the following conditions: 
	\begin{enumerate}
		\item $ \phi(x+y) \le \phi(x)+\phi(y). $
		\item $ \phi(xy) \ll \phi(x)\phi(y) $ where the implied constant is independent of $ x,y. $
	\end{enumerate}
\end{definition} 

Given $(\bA,\phi)$ as above and $R>0$ we write $B_R(\bA) := \{a\in \bA: \phi(a) \le R\}$ for the ball of radius $R.$ We say $(\bA,\phi)$ has \emph{linear growth} if for all $0<R<\infty$ the set $B_R(\bA)$ is finite and satisfies
$|B_{CR}(\bA)| \ll_{C} |B_R(\bA)|$ for any constant $C.$

\begin{lemma}[Point counting and dimension]\label{normal-analytic}
	Let $ \bA $ be an infinite integral domain equipped with a pseudo-norm $\phi$ such that $(\bA,\phi)$ has linear growth. Let $ \bk $ be the fraction field of $\bA$ and $ V\subset \bar{\bk}^s $ an algebraic set defined over $ \bA $ by $ m $ homogeneous equations of degree $ n. $ For $ R>0 $ set
	\[ 
	N_R = |V\cap (B_R(\bA))^s|.
	 \]
	 Given $ D>0, $ there exists $ R_0(s,m,n,D) $ (but independent of $ V $!) such that if for some $ R>R_0 $ we have $ N_R \ge D|B_R(\bA)|^{s-t}  $ then $ \textnormal{codim}(V) \le t. $ 
\end{lemma}

\begin{proof}[Proof of lemma \ref{normal-analytic}]
	 Suppose, to get a contradiction, that $ \textnormal{codim}(V) =t' > t. $ By infinitude of $\bA,$ we can apply lemma \ref{normal} to get a constant $C=C(s,m,n)$ and a matrix $M\in \mathcal{M}_{(s-t')\times s}(\bA)$ with entries in $B_C(\bA)$ such that $M:V\to \bar{\bk}^{s-t'}$ has fibers of size $ \le C. $ This allows us to bound $N_R$
	 \[
	 N_R \le C\cdot |M((B_R(\bA))^s)| \le C\cdot |B_{C'R}(\bA)|^{s-t'}\ll_{s,m,n} |B_R(\bA)|^{s-t-1},
	 \]
	 where $C'= C'(s,m,n). $ The second and third inequalities use the fact that $\phi$ is a pseudo-norm and that $(\bA,\phi)$ has linear growth, respectively. If \linebreak
	 $N_R\ge D|B_R(\bA)|^{s-t}$ then we combine with the above inequality to get
	 \[
	 D\le D'|B_R(\bA)|^{-1}
	 \] 
	 where $D' = D'(s,m,n)>0.$ By infinitude of $\bA,$ this is impossible for\linebreak
	 $R\ge R_0(s,m,n,D)$, so we must have $\textnormal{codim}(V) \le t. $
\end{proof}

\begin{corollary}\label{points-to-rank}
    Let $ \bA $ be an infinite integral domain equipped with a pseudo-norm $\phi$ such that $(\bA,\phi)$ has linear growth. Let $\bk$ be the fraction field of $\bA,$ $P:\bk^{s_1}\times\ldots\times \bk^{s_d}\to\bk$ a multilinear form and $D$ a positive constant. Write $s =  s_1+\ldots+s_{d-1}.$ Then there exists $R_0(s,d,D)$ independent of $P$ such that if we have 
    \[
    |Z_P(\bk)\cap (B_R(\bA))^s| \ge D|B_R(A)|^{s-t} 
    \]
    for some $R\ge R_0$ then $\prk_\bk (P) \le 2^{d-1} \cdot t.$
\end{corollary}

\begin{proof}
    Combine lemma \ref{normal-analytic} and  theorem \ref{rank-sing}.
\end{proof}

Corollary \ref{points-to-rank} is the result that we will actually use in our proof of proposition \ref{rank-sing-families}, the other lemmas were preliminary.

\begin{proof}[Proof of proposition \ref{rank-sing-families}]
	Let $s = s_1+\ldots+s_{d-1}.$ First we claim that our hypothesis implies $Z_{\mathcal{P}}(\bk^s) = \bk^s.$ Indeed, suppose this isn't the case. Then there is some $\bx\in \bk^s\setminus Z_{\mathcal{P}}(\bk^s),$ meaning $P_i(\bx,\cdot)$ are linearly independent and $P_i(\bx,\cdot) = c_i$ has a solution in $\bk^{s_d}$. The proof of the proposition will be similar for both types of fields. \linebreak
	\textbf{Number fields:} Let $ \bA $ be the ring of integers of $ \bk. $ Recalling that $ \bA \cong \mZ^m$ as $ \mZ $-modules, we define a pseudo-norm on $ \bA $ with linear growth by
	\[
	\phi(n_1,\ldots,n_m) = \max_i |n_i|.
	\]
	Assume without loss of generality that the collection $ \mathcal{P} $ has coefficients in $ \bA. $ Now, suppose $ \bx\in Z_\mathcal{P}(\bk)\cap B_R(\bA)^s. $ Then the matrix $ M = (m_{i,j}),$ where \linebreak $m_{i,j} = P_j(\bx_1,\ldots,\bx_{d-1},e_i), $ is of rank $ \le n-1 $ and has $ \phi(m_{i,j}) \ll R^{d-1} $ for all $i,j.$ By a standard linear algebra argument, we obtain a small vector in the kernel of $M.$
	\begin{claim}\label{small-sol}
		Suppose $ (\bA,\phi) $ is an integral domain equipped with a pseudo-norm. Suppose $ M\in M_{m\times n}(\bA) $ has $ \rk(M) \le n-1 $ and $ \phi(m_{i,j}) \le T $ for all $i,j.$ Then there exists $ 0\neq a\in \bA^n $ such that $ Ma = 0 $ and $ \phi(a_i) \ll_n T^{n-1} $ for all $i.$ 
	\end{claim}  
	\begin{proof}
		Suppose without loss of generality that the first $ j\le n-1 $ rows of $ M $ are a $\bk$-basis for $ \textnormal{row}(M). $ We obtain a new matrix by deleting the other rows and then adding $ n-j $ rows of standard basis vectors to obtain an invertible $ n\times n $ matrix which we call $ M'. $ Let $D = \det M'.$ Then Cramer's rule gives a solution $ 0\neq a\in \bA^n $ to the equation $ M'a = (0,\ldots,0,D) $ which satisfies $ \phi(a_i) \ll_n T^{n-1} $ for all $i.$ 
	\end{proof}
	Therefore, there is a nontrivial solution $ a \in \bA^n  $ to the equation $ Ma = 0 $ which satisfies $ \phi(a_i) \ll R^{(d-1)(n-1)} $ for all $i.$ The number of possibilities for such an $ a $ is $ \ll R^{mn(d-1)(n-1)}. $ Hence there must be some nontrivial linear combination $ P = a_1P_1+\ldots+a_nP_n $ with
	\[ 
	Z_P (\bk) \cap (B_R(\bA))^s \gg R^{m[s-(d-1)n^2]}.
	\]
	If $R$ is sufficiently large, we can apply corollary \ref{points-to-rank} to obtain 
	$$ \prk_{\bk}(P) \le  2^{d-1} (d-1)n^2.$$
	
	\textbf{Finite separable extensions of}  $ \mF_q(t) $:  Let $ \bA $ be the integral closure of $ \mF_q[t] $ in $ \bk. $ Then $ \bA \cong \mF_q[t]^m $ as a $ \mF_q[t] $-module, with the same proof as for rings of integers in number fields. Define a pseudo-norm on $ \mF_q[t] $ with linear growth by 
	$ \phi(f) = q^{\deg(f)} $ and extend it to $ \bA $ by $ \phi(f_1,\ldots,f_m) = \max_i \phi(f_i). $ Assume without loss of generality that $ \mathcal{P} $ has coefficients in $ \bA. $ From here the proof proceeds in a similar manner: If $  x\in Z_\mathcal{P}(\bk)\cap B_{q^R}(\bA)^s $ then the matrix $ M = (m_{i,j}),$ where $m_{i,j} = P_j(\bx_1,\ldots,\bx_{d-1},e_i), $ is of rank $ \le n-1 $ and has $ \phi(m_{i,j}) \ll q^{R(d-1)} $ for all $i,j.$ By claim \ref{small-sol} there is a nontrivial solution $ a \in \bA^n  $ to the equation $ Ma = 0 $ which satisfies $ \phi(a_i) \ll q^{R(d-1)(n-1)} $ for all $i.$ The number of possibilities for such an $ a $ is $ \ll q^{mnR(d-1)(n-1)}. $ Hence there must be some nontrivial  linear combination $ P = a_1P_1+\ldots+a_nP_n $ with
	\[ 
		Z_P (\bk) \cap (B_{q^R}(\bA))^s \gg q^{mR[s-(d-1)n^2]}.
	\]
	If $R$ is sufficiently large, we can apply corollary \ref{points-to-rank} to obtain 
	$$ \prk_{\bk}(P) \le  2^{d-1}(d-1)n^2. $$ 
\end{proof}

Now we can finish the proof of theorem \ref{universal}, and hence of theorems \ref{main-ml} and \ref{main-multiple}.

\begin{proof}
We will show that theorem \ref{universal} holds for number fields and finite separable extensions of $\mF_q(t)$ with  constants $C = 2^{d-1}(d-1), D=2.$ Suppose $P_1,\ldots,P_n$ satisfy $\prk_\bk (P_1,\ldots,P_n) > 2^{d-1}(d-1)(nt^d)^2.$ Recall that we proved that this implies $\prk_\bk(P^l_{j_1,\ldots,j_d}) > 2^{d-1}(d-1)(nt^d)^2.$ By proposition $\ref{rank-sing-families}$ the system of equations 
\[
P^l_{j_1,\ldots,j_d}(A_1,\ldots, A_d) = r^l_{j_1,\ldots,j_d}
\]
from the end of section $ 3 $ has a $\bk$-solution as desired. Recall that for the constants $ \tilde{A} , \tilde{B} $ of theorem \ref{main-ml} we can take
\[
    \tilde{A} =  C = 2^{d-1}(d-1),\ 
    \tilde{B} = dD = 2d,
\]

and for the constants $A,B$ of theorem \ref{main-multiple} we can take
\[
A = \tilde{A} \binom{d}{\lfloor d/2 \rfloor}^{\tilde{B}} = 2^{d-1}(d-1) \binom{d}{\lfloor d/2 \rfloor}^{2d}, \  
    B = \Tilde{B} = 2d.
\]
\end{proof}
	
\bibliography{papers}
\bibliographystyle{plain}

\end{document}